\documentclass[11pt]{amsart}
\usepackage[margin=1in]{geometry}

\RequirePackage[dvipsnames,usenames]{xcolor}
\usepackage{amsmath}
\usepackage{amscd}
\usepackage{tikz-cd}
\usepackage{amssymb}
\usepackage[mathscr]{euscript}
\usepackage{enumerate}
\usepackage{hyperref}

\DeclareMathOperator{\Hom}{Hom}

\DeclareMathOperator{\Ext}{Ext}
\DeclareMathOperator{\Tor}{Tor}

\DeclareMathOperator{\Ann}{Ann}

\DeclareMathOperator{\hgt}{ht}
\DeclareMathOperator{\gr}{gr}

\usepackage{amsthm}

\theoremstyle{plain}
\newtheorem{thm}{Theorem}[section]
\newtheorem{bigthm}{Theorem}

\newtheorem{prop}[thm]{Proposition}
\newtheorem{lem}[thm]{Lemma}
\newtheorem{cor}[thm]{Corollary}

\theoremstyle{definition}

\theoremstyle{remark}
\newtheorem{remark}[thm]{Remark}

\newcommand{\mm}{\mathfrak{m}}

\newcommand{\D}{\mathscr{D}}

\newcommand{\cD}{\widehat{\mathscr{D}}}
\newcommand{\cR}{\widehat{R}}
\newcommand{\cM}{\widehat{M}}
\newcommand{\cN}{\widehat{N}}
\newcommand{\cmm}{\widehat{\mathfrak{m}}}

\newcommand{\dR}{\mathrm{dR}}

\newcommand{\wh}{\widehat}

\begin{document}

\title{ON completion of graded $\D$-modules}
\author{Nicholas Switala \and Wenliang Zhang}
\address{Department of Mathematics, Statistics, and Computer Science \\ University of Illinois at Chicago \\ 322 SEO (M/C 249) \\ 851 S. Morgan Street \\ Chicago, IL 60607}
\email{nswitala@uic.edu, wlzhang@uic.edu}
\thanks{The first author gratefully acknowledges NSF support through grant DMS-1604503. The second author is partially supported by the NSF through DMS-1606414 and CAREER grant DMS-1752081.}

\begin{abstract}
Let $R = k[x_1, \ldots, x_n]$ be a polynomial ring over a field $k$ of characteristic zero and $\cR$ be the formal power series ring $k[[x_1, \ldots, x_n]]$. If $M$ is a $\D$-module over $R$, then $\cR \otimes_R M$ is naturally a $\D$-module over $\cR$. Hartshorne and Polini asked whether the natural maps $H^i_{\dR}(M)\to H^i_{\dR}(\cR \otimes_R M)$ (induced by $M\to \cR \otimes_R M$) are isomorphisms whenever $M$ is graded and holonomic. We give a positive answer to their question, as a corollary of the following stronger result. Let $M$ be a finitely generated graded $\D$-module: for each integer $i$ such that $\dim_kH^i_{\dR}(M)<\infty$, the natural map $H^i_{\dR}(M)\to H^i_{\dR}(\cR \otimes_R M)$ (induced by $M\to \cR \otimes_R M$) is an isomorphism.
\end{abstract}

\maketitle

\section{Introduction}\label{introduction}

Let $k$ be a field of characteristic zero. Let $R$ be the polynomial ring in $n$ variables over $k$, and let $\cR$ be the formal power series ring in $n$ variables over $k$. Consider the rings $\D = \D(R,k)$ (resp. $\cD = \D(\cR, k)$) of $k$-linear differential operators on $R$ (resp. $\cR$). In this paper we investigate the behaviors of de Rham cohomology of graded $\D$-modules under completion, which is motivated by a question posed in \cite[p. 18]{HartshornePolini} by Hartshorne and Polini. In \cite{HartshornePolini}, Hartshorne and Polini investigate the $\D$-module structure of local cohomology modules of $R$ supported in homogeneous ideals: in particular, their de Rham cohomology spaces. They use some technical results on $\cD$-modules, due to van den Essen, that have no analogues over the polynomial ring. Since their motivation lies with the polynomial ring, they investigate the preservation of de Rham cohomology under the operation of completion.

If $M$ is a left $\D$-module, there is a natural left $\cD$-module structure on the $\cR$-module $\cR \otimes_R M$ (see section \ref{completion} below). We can therefore compare the de Rham cohomology of $M$ with that of $\cR \otimes_R M$. In fact, there are natural maps
\[
H^i_{\dR}(M) \rightarrow H^i_{\dR}(\cR \otimes_R M)
\]
of $k$-spaces, induced by the obvious natural maps $M\to \cR\otimes_R M$, for all $i \geq 0$. Hartshorne and Polini prove that these maps are isomorphisms in the case when $M=H^j_I(R)$ with $I$ a homogeneous ideal of $R$, but need not be isomorphisms in general (even for \emph{holonomic} $M$):

\begin{thm}[Hartshorne-Polini]
\label{HP results on completion}
Let $R$, $\cR$, $\D$, and $\cD$ be as above.
\begin{enumerate}[(a)]
\item \cite[Theorem 6.2]{HartshornePolini} If $I \subseteq R$ is a homogeneous ideal, then for all $i,j \geq 0$, the completion map
\[
H^i_{\dR}(H^j_I(R)) \rightarrow H^i_{\dR}(\cR \otimes_R H^j_I(R))
\]
is an isomorphism of $k$-spaces.
\item \cite[Example 6.1]{HartshornePolini} Let $n=1$. There exists a holonomic left $\D$-module $M$ such that the completion map
\[
H^i_{\dR}(M) \rightarrow H^i_{\dR}(\cR \otimes_R M)
\]
is not an isomorphism of $k$-spaces for $i=0$ or $i=1$. (In fact, the map for $i=0$ is not surjective and the map for $i=1$ is not injective.)
\end{enumerate}
\end{thm}

In the situation of Theorem \ref{HP results on completion}(a), since the ideal $I$ is homogeneous, the $\D$-modules $H^j_I(R)$ are \emph{graded} (which means here that the partial derivatives $\partial_i$ act as $k$-linear maps of degree $-1$). On the other hand, the example of Theorem \ref{HP results on completion}(b) is not a graded $\D$-module ($\partial_1$ acts as a $k$-linear map of degree $2$). Hartshorne and Polini ask \cite[p. 18]{HartshornePolini} whether the completion maps on de Rham cohomology are isomorphisms if $M$ is a \emph{graded} holonomic $\D$-module. 

The main result of this paper is the following:

\begin{bigthm}[Theorem \ref{dR coh of completion of graded}]
\label{main theorem}
Let $M$ be a finitely generated graded left $\D$-module. 
\begin{enumerate}[(a)]
\item The natural map 
\[
H^i_{\dR}(M) \rightarrow H^i_{\dR}(\cR \otimes_R M)
\]
(induced by $M\to \cR\otimes_R M$) is injective for all $i\geq 0$.
\item For each integer $i \geq 0$ such that $\dim_kH^i_{\dR}(M)<\infty$, the natural map
\[
H^i_{\dR}(M) \rightarrow H^i_{\dR}(\cR \otimes_R M)
\]
is an isomorphism of $k$-spaces.
\end{enumerate}
\end{bigthm}

As Theorem \ref{HP results on completion}(b) shows, even Theorem \ref{main theorem}(a) may fail if $M$ is not graded. The hypothesis of finite-dimensionality in Theorem \ref{main theorem}(b) is also necessary: see Remark \ref{fin dim necessary} below.

Note that in the statement of Theorem \ref{main theorem} the graded $\D$-module $M$ is {\it not} assume to be holonomic. However, if $M$ is a graded \emph{holonomic} $\D$-module, then all of its de Rham cohomology spaces are finite-dimensional, so we immediately obtain a positive answer to Hartshorne and Polini's question:

\begin{cor}\label{answer to HP question}
Let $M$ be a graded holonomic left $\D$-module. For all $i \geq 0$, the completion map
\[
H^i_{\dR}(M) \rightarrow H^i_{\dR}(\cR \otimes_R M)
\]
is an isomorphism of $k$-spaces.
\end{cor}

After briefly recalling some preliminary materials and fixing notation in section \ref{preliminaries}, we study the completion operation on $\D$-modules in section \ref{completion}. Hartshorne and Polini observe that if $M$ is a local cohomology module of $R$ (and therefore holonomic), $\cR \otimes_R M$ is again holonomic; we prove this statement for arbitrary holonomic $M$ in this section. Finally, in section \ref{derham}, we give a proof of Theorem \ref{main theorem}. We conclude by outlining a shorter proof that, if $M$ is a graded holonomic left $\D$-module, then $H^i_{\dR}(M)$ and $H^i_{\dR}(\cR \otimes_R M)$ have the same dimension (this shorter proof has the deficiency that it says nothing about whether the natural completion maps are isomorphisms).

\section{Preliminaries}\label{preliminaries}

In this section, we collect some preliminary materials on (graded) $\D$-modules and de Rham cohomology. Much of the basic material is recalled already in Hartshorne and Polini's \cite{HartshornePolini} as well as the authors' earlier \cite{SwitalaZhangGradedDual}; we will assume the reader is familiar either with the introductory sections of those papers or with the basic reference \cite{Bjork}. In particular, we will assume the reader is familiar with the notion of a \emph{holonomic} $\D$-module.

\subsection{Notation} Throughout this paper, $k$ is a field of characteristic zero. We denote by $R = k[x_1, \ldots, x_n]$ the \emph{polynomial} ring over $k$ in the variables $x_1, \ldots, x_n$ for some $n \geq 1$, and by $\cR = k[[x_1, \ldots, x_n]]$ the \emph{formal power series} ring over $k$ in the same variables. Observe that $\cR$ is the $\mm$-adic completion of $R$, where $\mm \subseteq R$ is the maximal ideal generated by $x_1, \ldots, x_n$.

Objects without ``hats'' will be associated with the ring $R$, and the corresponding objects with ``hats'' will be associated with $\cR$. Therefore, we will write $\D$ for the ring $\D(R,k)$ of $k$-linear differential operators on $R$ and $\cD$ for the ring $\D(\cR,k)$ of $k$-linear differential operators on $\cR$. Recall that $\D$ (resp. $\cD$) is generated over $R$ (resp. over $\cR$) by the partial differentiation operators $\partial_1, \ldots, \partial_n$. A ``$\D$-module'' $M$ will always be assumed to be a \emph{left} module, unless stated otherwise, and similarly for $\cD$-modules. 

If $M$ is a $\D$-module, we denote by $\cM$ its \emph{completion} $\cR \otimes_R M$ (see section \ref{completion}). Observe that if $M$ is not finitely generated as an $R$-module, this object need not be isomorphic to the $\mm$-adic completion of $M$; the notation $\cM$ and term ``completion'' are therefore somewhat abusive. Likewise, $\cD$ is not being regarded as an adic completion of the non-commutative ring $\D$. 

\subsection{The de Rham complex} Given any $\D$-module $M$, we can define its \emph{de Rham complex} $\Omega^{\bullet}(M)$, whose objects are $\D$-modules but whose differentials are merely $k$-linear.  It is defined as follows \cite[\S 1.6]{Bjork}: for $0 \leq i \leq n$, $\Omega^i(M)$ is a direct sum of $n \choose i$ copies of $M$, indexed by $i$-tuples $1 \leq j_1 < \cdots < j_i \leq n$.  The summand corresponding to such an $i$-tuple will be written as $M \, dx_{j_1} \wedge \cdots \wedge dx_{j_i}$. The $k$-linear differentials $d^i: \Omega^i(M) \rightarrow \Omega^{i+1}(M)$ are defined by 
\[
d^i(m \,dx_{j_1} \wedge \cdots \wedge dx_{j_i}) = \sum_{s=1}^n \partial_s(m)\, dx_s \wedge dx_{j_1} \wedge \cdots \wedge dx_{j_i},
\]
with the usual exterior algebra conventions for rearranging the wedge terms, and extended by linearity to the direct sum. The cohomology objects $h^i(\Omega^{\bullet}(M))$, which are $k$-spaces, are called the \emph{de Rham cohomology spaces} of the $\D$-module $M$, and are denoted $H^i_{\dR}(M)$. If $N$ is a $\cD$-module, its de Rham complex $\Omega^{\bullet}(N)$, with cohomology spaces $H^i_{\dR}(N)$, has exactly the same definition. The objects of this complex are $\cD$-modules, but its differentials are again merely $k$-linear.

Part (a) of the following theorem is standard (see \cite[Theorem 1.6.1]{Bjork}); part (b) is due to van den Essen \cite[Proposition 2.2]{vdEcokernelsII}:

\begin{thm}\label{dRfindim}
\begin{enumerate}[(a)]
\item Let $M$ be a holonomic $\D$-module. The de Rham cohomology spaces $H^i_{\dR}(M)$ are finite-dimensional over $k$ for all $i \geq 0$.
\item Let $N$ be a holonomic $\cD$-module. The de Rham cohomology spaces $H^i_{\dR}(N)$ are finite-dimensional over $k$ for all $i \geq 0$.
\end{enumerate}
\end{thm} 

\subsection{Graded $\D$-modules} We give the polynomial ring $R = k[x_1, \ldots, x_n]$ its standard grading, {\it i.e.}, $\deg(x_i)=1$ for $1 \leq i \leq n$ and $\deg(c)=0$ for $c\in k$. Let $M$ be a (left) $\D$-module whose underlying $R$-module is given a grading $M = \oplus_{l \in \mathbb{Z}} M_l$ (meaning that $R_i \cdot M_j \subseteq M_{i+j}$ for all $i,j$, where $R_i$ is the degree-$i$ component of $R$). We say that $M$ is a \emph{graded $\D$-module} if for all $l \in \mathbb{Z}$ and $1 \leq i \leq n$, we have $\partial_i(M_l) \subseteq M_{l-1}$. There is an entirely analogous notion of graded \emph{right} $\D$-module. If $M$ is a finitely generated graded $\D$-module, $M$ admits a resolution by finite free graded $\D$-modules.

\subsection{Transposition} There is a natural \emph{transposition} operation that converts left $\D$-modules to right $\D$-modules and \emph{vice versa}. (This is not described in the reference \cite{Bjork}; see \cite[\S 16]{Coutinho} instead.) The \emph{standard transposition} $\tau: \D \rightarrow \D$ is defined by
\[
\tau(f \partial_1^{i_1} \cdots \partial_n^{i_n}) = (-1)^{i_1 + \cdots + i_n}\partial_1^{i_1} \cdots \partial_n^{i_n}f
\]
for all $f \in R$, extended to all of $\D$ by $k$-linearity (observe that the same operation makes sense for formal power series). If $M$ is a right $\D$-module, the \emph{transpose} $M^{\tau}$ of $M$ is the left $\D$-module defined as follows: we have $M^{\tau} = M$ as Abelian groups, and the left $\D$-action $\ast$ on $M^{\tau}$ is given by $\delta \ast m = m \cdot \tau(\delta)$ for all $\delta \in \D$ and $m \in M (=M^{\tau})$. The transpose of a \emph{graded} left $\D$-module is a \emph{graded} right $\D$-module, and conversely.

\section{The completion functor for $\D$-modules}\label{completion}

Let $M$ be a $\D$-module. As observed by Hartshorne and Polini in \cite[\S 6]{HartshornePolini}, the $\cR$-module $\cR \otimes_R M$ can be given a natural structure of $\cD$-module. (Recall from section \ref{preliminaries} that we will abuse notation by writing $\cM$ for this $\cR$-module.) In this section, we study the basic properties of the functor $M \mapsto \cM$ from $\D$-modules to $\cD$-modules.  

As a special case of \cite[Lemma 1.2.1]{Hotta}, we have the following recipe for prescribing $\D$-module (resp. $\cD$-module) structures on $R$-modules (resp. $\cR$-modules):

\begin{lem}\label{extending R-structures}
Let $M$ be an $R$-module. To give $M$ a structure of $\D$-module extending the given $R$-module structure is the same as to give \emph{pairwise commuting} $k$-linear maps $\partial_i: M \rightarrow M$ such that, for all $1 \leq i \leq n$, all $r \in R$, and all $m \in M$, we have $\partial_i(rm) = r\partial_i(m) + \partial_i(r)m$. (The analogous statement for $\cR$-modules $\cM$ also holds.)
\end{lem}

We will, however, write $\partial_i m$ or $\partial_i \cdot m$ instead of $\partial_i(m)$, reserving the notation $\partial_i(-)$ for application of $\partial_i$ to elements of $R$ or $\cR$.

\begin{prop}\label{Hartshorne structure}
Let $M$ be a $\D$-module, and let $\cM$ be the $\cR$-module $\cR \otimes_R M$. For all $i$ and all pure tensors $\wh{s} \otimes m \in \cM$, define
\[
\partial_i \cdot (\wh{s} \otimes m) = \partial_i(\wh{s}) \otimes m + \wh{s} \otimes \partial_i \cdot m.
\]
Then (extending by $\cR$-linearity) we obtain a structure of $\cD$-module on $\cM$. Furthermore, using this $\cD$-module structure, the operation $M \mapsto \cM$ is a functor from $\D$-modules to $\cD$-modules.
\end{prop}

\begin{proof}
By Lemma \ref{extending R-structures}, it suffices to check that for all $\wh{r} \in \cR$ and $1 \leq i \leq n$, the actions of $\partial_i \wh{r} - \wh{r} \partial_i$ and $\partial_i(\wh{r})$ on the pure tensor $\wh{s} \otimes m$ coincide. Indeed, we have
\begin{align*}
(\partial_i \wh{r} - \wh{r} \partial_i) \cdot (\wh{s} \otimes m) &= \partial_i \wh{r} \cdot (\wh{s} \otimes m) - \wh{r} \partial_i \cdot (\wh{s} \otimes m)\\
&= \partial_i \cdot (\wh{r}\wh{s} \otimes m) - \wh{r} \cdot (\partial_i(\wh{s}) \otimes m + \wh{s} \otimes \partial_i \cdot m)\\
&= \partial_i(\wh{r}\wh{s}) \otimes m + \wh{r}\wh{s} \otimes \partial_i \cdot m - \wh{r}\partial_i(\wh{s}) \otimes m - \wh{r}\wh{s} \otimes  \partial_i \cdot m\\
&= (\partial_i(\wh{r}\wh{s}) - \wh{r}\partial_i(\wh{s})) \otimes m\\
&= \partial_i(\wh{r})\wh{s} \otimes m,
\end{align*}
since $\partial_i$ is a derivation. It follows that $\cM$ is a $\cD$-module. For the functoriality, suppose that $\delta: M \rightarrow N$ is a map of $\D$-modules. We claim that
\[
\wh{\delta} = \mathrm{id}_{\cR} \otimes \delta: \cM \rightarrow \cN
\]
is a map of $\cD$-modules. Since $\wh{\delta}$ is clearly $\cR$-linear, it is enough to show that 
\[
\wh{\delta}(\partial_i \cdot (\wh{s} \otimes m)) = \partial_i \cdot \wh{\delta}(\wh{s} \otimes m)
\]
for all $i$ and all pure tensors $\wh{s} \otimes m \in \cM$, for which we simply calculate:
\begin{align*}
\wh{\delta}(\partial_i \cdot (\wh{s} \otimes m)) &= \wh{\delta}(\partial_i(\wh{s}) \otimes m) + \wh{\delta}(\wh{s} \otimes \partial_i \cdot m)\\
&= \partial_i(\wh{s}) \otimes \delta(m) + \wh{s} \otimes \delta(\partial_i \cdot m)\\
&= \partial_i(\wh{s}) \otimes \delta(m) + \wh{s} \otimes \partial_i \cdot \delta(m)\\
&= \partial_i \cdot \wh{\delta}(\wh{s} \otimes m),
\end{align*}
using the $\D$-linearity of $\delta$. This completes the proof.
\end{proof}

The definition of the $\cD$-structure on $\cM$ just given depends \emph{a priori} on the choice of coordinates $\{x_i, \partial_i\}$. There is an alternative, coordinate-free definition which we will also find useful below. 

\begin{prop}\label{tensor structure}
Let $M$ be a $\D$-module. Form the tensor product $\cD \otimes_{\D} M$ using the right $\D$-module structure on $\cD$ defined via right multiplication by the subring $\D \subseteq \cD$, and regard this tensor product as a left $\cD$-module via left multiplication on the first tensor factor.

There is an isomorphism $\cD \otimes_{\D} M \cong \cR \otimes_R M$ of (left) $\cD$-modules, where the $\cD$-module structure on the left-hand side is the one just described, and the $\cD$-module structure on the right-hand side is the one given in Proposition \ref{Hartshorne structure}.
\end{prop}

\begin{proof}
Since $\cD$ is free as an $\cR$-module on the monomials in $\partial_1, \ldots, \partial_n$, there is a natural isomorphism $\cR \otimes_R \D \cong \cD$ as $\cR$-modules, given by $\wh{r} \otimes \delta \mapsto \wh{r}\delta$. This is also an isomorphism of right $\D$-modules, where the right $\D$-module structure on $\cR \otimes_R \D$ is given by right multiplication on the second tensor factor. It follows that we have isomorphisms
\[
\cR \otimes_R M \cong \cR \otimes_R (\D \otimes_{\D} M) \cong (\cR \otimes_R \D) \otimes_{\D} M \cong \cD \otimes_{\D} M
\]
of $k$-spaces, where the first two are the obvious canonical isomorphisms. It is clear that elements of $\cR$ act in the same way on both sides, and that the composite isomorphism (reading left to right) carries $\wh{r} \otimes_R m$ to $\wh{r} \otimes_{\D} m$ (and therefore $\wh{r} \otimes_R (\delta \cdot m)$ to $\wh{r}\delta \otimes_{\D} m$). Therefore, we need only check that the action of $\partial_i$ is respected. Given a pure tensor $\wh{r} \otimes_R m \in \cR \otimes_R M$, we have $\partial_i \cdot (\wh{r} \otimes_R m) = \partial_i(\wh{r}) \otimes_R m + \wh{r} \otimes_R \partial_i \cdot m$ by Proposition \ref{Hartshorne structure}, which is carried by the composite isomorphism to $\partial_i(\wh{r}) \otimes_{\D} m + \wh{r} \otimes_{\D} \partial_i \cdot m$. On the other hand, we have (using the $\cD$-module structure on $\cD \otimes_{\D} M$ defined in the present proposition)
\begin{align*}
\partial_i \cdot (\wh{r} \otimes_{\D} m) &= \partial_i \wh{r} \otimes_{\D} m\\ &= \partial_i(\wh{r}) \otimes_{\D} m + \wh{r}\partial_i \otimes_{\D} m \\ &= \partial_i(\wh{r}) \otimes_{\D} m + \wh{r} \otimes_{\D} \partial_i \cdot m,
\end{align*}
since $\partial_i$ is a derivation. The proof is complete.
\end{proof}

We next record some basic properties of the functor $M \mapsto \cM$. Note that by Proposition \ref{tensor structure}, we may view $\cM$ either as $\cR \otimes_R M$ or as $\cD \otimes_{\D} M$, whichever is more convenient. \emph{A priori}, the notation $\cD$ is ambiguous, referring on the one hand to the ring $\D(\wh{R}, k)$ and on the other hand to the $\wh{R}$-module $\wh{R} \otimes_R \D$ endowed with a $\D(\wh{R}, k)$-module structure. Part (b) of the following proposition removes this ambiguity.

\begin{prop}\label{completion omnibus}
\begin{enumerate}[(a)]
\item The functor $M \mapsto \cM$ from $\D$-modules to $\cD$-modules is exact.
\item The $\cD$-module $\cR \otimes_R \D$ is free of rank one.
\item Let $M$ be a $\D$-module, and let $F_{\bullet} \rightarrow M$ be a free resolution of $M$ by $\D$-modules. Then
\[
\cR \otimes_R F_{\bullet} \cong \cD \otimes_{\D} F_{\bullet}
\] 
is a free resolution of $\cM$ by $\cD$-modules (of the same ranks).
\end{enumerate}
\end{prop}

\begin{proof}
Part (a) follows from the fact that $\cR$ is a flat $R$-module. By Proposition \ref{tensor structure}, we have $\cR \otimes_R \D \cong \cD \otimes_{\D} \D = \cD$ as $\cD$-modules, proving part (b). Finally, part (c) follows immediately from parts (a) and (b) and the commutativity of tensor product with direct sum.
\end{proof}

Less obviously, the completion functor preserves the property of holonomicity. (Hartshorne and Polini observed this already in the case of local cohomology modules of $R$.) To prove this, it is convenient to use the definition of the completion functor given in Proposition \ref{Hartshorne structure}.

\begin{prop}\label{completion is holonomic}
If $M$ is a holonomic $\D$-module, then $\cM$ is a holonomic $\cD$-module.
\end{prop}

Proposition \ref{completion is holonomic} is a corollary of the following stronger statement:

\begin{prop}\label{dimension equality}
Let $M$ be a finitely generated $\D$-module. We have $d_{\D}(M) = d_{\cD}(\cM)$, where $d$ denotes $\D$- (resp. $\cD$-) module dimension.
\end{prop}


\begin{proof}
We let $\{F_l \D\}_{l \geq 0}$ (resp. $\{F_l \cD\}_{l \geq 0}$) denote the order filtration on $\D$ (resp. $\cD$), and $\gr \D$ (resp. $\gr \cD$) the \emph{commutative} associated graded ring with respect to this filtration. We have $\gr \D \cong R[\xi_1, \ldots, \xi_n]$ (resp. $\gr \cD \cong \cR[\xi_1, \ldots, \xi_n]$) where $\xi_i$ is the image of $\partial_i$ in $F_1\D/F_0\D$ (resp. $F_1 \cD/F_0 \cD$).

Choose a good filtration $\{G_p M\}_{p \geq 0}$ of $M$, and let $\gr M$ be the associated (finitely generated) graded $\gr \D$-module. For each $p$, write $G_p \cM$ for $\wh{G_p M} = \cR \otimes_R G_pM$, which we identify with an $\cR$-submodule of $\cM$. Clearly $\cup_{p \geq 0} G_p \cM = \cM$. Moreover, for all $i$ and $p$ and all $m \in G_p M$, we have
\[
\partial_i \cdot (\wh{r} \otimes m) = \partial_i(\wh{r}) \otimes m + \wh{r} \otimes \partial_i \cdot m \in G_p \cM + G_{p+1} \cM = G_{p+1} \cM,
\]
using the definition of Proposition \ref{Hartshorne structure}, so that the family $\{G_p \cM\}_{p \geq 0}$ makes $\cM$ into a filtered $\cD$-module. By the flatness of $\cR$ over $R$, we see that $\gr \cM \cong \cR \otimes_R \gr M$ as $\cR$-modules. Under this identification, if $\overline{m}$ is the class of $m$ in $G_p M/G_{p-1}M \subseteq \gr M$, we see from the displayed equation that $\xi_i \cdot (\wh{r} \otimes \overline{m}) = \wh{r} \otimes \xi_i \cdot \overline{m}$; that is, the $\xi_i$ act by multiplication on the second tensor factor. It follows from this that $\gr \cM \cong \gr \cD \otimes_{\gr \D} \gr M$ as $\gr \cD$-modules. Since $\gr M$ is finitely generated over $\gr \D$, $\gr \cM$ is finitely generated over $\gr \cD$, so this filtration is good. Finally, let $J = \Ann_{\gr \D} \gr M$. Since $\gr M$ is a finitely generated $\gr \D$-module and the inclusion $\gr \D \subseteq \gr \cD$ is faithfully flat, the annihilator of $\gr \cM \cong \gr \cD \otimes_{\gr \D} \gr M$ in $\gr \cD$ is the ideal $J \cdot \gr \cD$. By the faithful flatness, we have $\hgt_{\gr \D} J = \hgt_{\gr \cD} J \cdot \gr \cD$, and these respective heights are by definition the desired dimensions, completing the proof.
\end{proof}

\section{De Rham cohomology and completion}\label{derham}

Let $M$ be a $\D$-module. If we regard the $\D$-module $\cM = \cR \otimes_R M$ as a $\D$-module by restriction of scalars, then the $R$-linear map $\kappa: M \rightarrow \cR \otimes_R M$ defined by $\kappa(m) = 1 \otimes m$ is in fact $\D$-linear: we have 
\[
\partial_i \cdot \kappa(m) = \partial_i(1) \otimes m + 1 \otimes \partial_i \cdot m = 1 \otimes \partial_i \cdot m = \kappa(\partial_i \cdot m)
\]
for $1 \leq i \leq n$. The map $\kappa$ induces a morphism of complexes of $k$-spaces
\[
\Omega^{\bullet}(M) \rightarrow \Omega^{\bullet}(\cM)
\]
by simply applying $\kappa$ to each summand of each object of the complex $\Omega^{\bullet}(M)$, and therefore induces maps
\[
\kappa^i: H^i_{\dR}(M) \rightarrow H^i_{\dR}(\cM)
\]
of $k$-spaces for all $i \geq 0$. The goal of this section is to prove the following result, our main theorem:

\begin{thm}\label{dR coh of completion of graded}
Let $M$ be a finitely generated graded $\D$-module. 
\begin{enumerate}[(a)]
\item The natural map
\[
\kappa^i: H^i_{\dR}(M) \rightarrow H^i_{\dR}(\cM)
\]
is injective for all $i\geq 0$.
\item For each integer $i \geq 0$ such that $\dim_kH^i_{\dR}(M)<\infty$, the natural map 
\[
\kappa^i: H^i_{\dR}(M) \rightarrow H^i_{\dR}(\cM)
\] 
is an isomorphism of $k$-spaces.
\end{enumerate}
\end{thm}

\begin{remark}\label{fin dim necessary}
The hypothesis of finite-dimensionality in Theorem \ref{dR coh of completion of graded}(b) is necessary. For example, $\D$ itself is a finitely generated graded $\D$-module. We have $H^n_{\dR}(\D) \cong R$ and $H^n_{\dR}(\cD) \cong \cR$, and the natural map $R \rightarrow \cR$ is injective but not surjective.
\end{remark}


In order to prove Theorem \ref{dR coh of completion of graded}, we will identify the de Rham cohomology of $M$ with certain $\Tor$ groups. In Proposition \ref{completion of Tor}, we prove an analogue of Theorem \ref{dR coh of completion of graded} for these $\Tor$ groups. Finally, in Proposition \ref{prop: de Rham Tor compatible with completion}, we construct a commutative diagram enabling us to deduce Theorem \ref{dR coh of completion of graded} from Proposition \ref{completion of Tor}.

\begin{prop}\label{completion of Tor}
Let $M$ be a finitely generated graded $\D$-module. 
\begin{enumerate}[(a)]
\item There are natural maps
\[
\iota_j: \Tor^{\D}_j(R^{\tau}, M) \rightarrow \Tor^{\cD}_j(\cR^{\tau}, \cM)
\]
induced by $M\to \cM$, that are injective for all $j\geq 0$.
\item Furthermore, for each $j \geq 0$ such that $\dim_k\Tor^{\D}_j(R^{\tau}, M)<\infty$, the natural map
\[
\iota_j: \Tor^{\D}_j(R^{\tau}, M) \rightarrow \Tor^{\cD}_j(\cR^{\tau}, \cM)
\] 
is an isomorphism of $k$-spaces.
\end{enumerate}
\end{prop}

\begin{proof}
Choose a graded free resolution $F_{\bullet}$ of $M$ as a $\D$-module. Since $M$ is finitely generated over $\D$, we may assume that each $F_j$ is \emph{finite} free, but possibly with shifts in the grading. That is, we have $F_j = \oplus_{l=1}^{\beta_j}\D(\gamma_{l,j})$ for some $\beta_j \geq 0$ and some integers $\gamma_{l,j}$. Then $\Tor^{\D}_j(R^{\tau}, M) = h_j(R^{\tau} \otimes_{\D} F_{\bullet})$ by definition. The maps $F_j \rightarrow F_{j-1}$ in the complex $F_{\bullet}$ are given by multiplication by $\beta_{(j-1)} \times \beta_j$ matrices $B_j$ whose entries are homogeneous elements of $\D$. The $k$-space $R^{\tau} \otimes_{\D} \D$ is simply $R$, under the identification $r \otimes 1 \mapsto r$, and if $\delta \in \D$ acts on $\D$ via left multiplication, then $\mathrm{id}_{R^{\tau}} \otimes_{\D} \D$ corresponds under this identification to $\tau(\delta): R \rightarrow R$. Passing to direct sums, we see that the complex $R^{\tau} \otimes_{\D} F_{\bullet}$ is isomorphic to
\[
R^{\beta_{\bullet}} = (\cdots \rightarrow \oplus_{l=1}^{\beta_2}R(\gamma_{l,2})) \xrightarrow{\tau(B_2)} \oplus_{l=1}^{\beta_1}R(\gamma_{l,1}) \xrightarrow{\tau(B_1)} \oplus_{l=1}^{\beta_0}R(\gamma_{l,0}) \rightarrow 0),
\]
where $\tau(B_j)$ denotes the matrix whose entries, still homogeneous elements of $\D$, are the transposes of the entries of $B_j$.

The completion $\widehat{F_{\bullet}}$ is a free resolution of $\cM$ as a $\cD$-module (with $\widehat{F_j} \cong \cD^{\beta_j}$) and $\Tor^{\cD}_j(\cR^{\tau}, \cM) = h_j(\cR^{\tau} \otimes_{\cD} \widehat{F_{\bullet}})$, again by definition. The matrices defining the differentials in the complex $\widehat{F_{\bullet}}$ are the same as those in $F_{\bullet}$ (that is, all entries are homogeneous elements in the subring $\D \subseteq \cD$), so that the complex $\cR^{\tau} \otimes_{\cD} \widehat{F_{\bullet}}$ is isomorphic to
\[
\cR^{\beta_{\bullet}} = (\cdots \rightarrow \cR^{\beta_2} \xrightarrow{\tau(B_2)} \cR^{\beta_1} \xrightarrow{\tau(B_1)} \cR^{\beta_0} \rightarrow 0),
\]
which contains $R^{\beta_{\bullet}}$ as a subcomplex. The natural map $\iota_{\bullet}: R^{\beta_{\bullet}} \rightarrow \cR^{\beta_{\bullet}}$ induces maps $\Tor^{\D}_j(R^{\tau}, M) \rightarrow \Tor^{\cD}_j(\cR^{\tau}, \cM)$ of $k$-spaces for all $j \geq 0$, which we again denote $\iota_j$.

First we show that $\iota_j$ is injective on homology for all $j \geq 0$. Let $z \in \oplus_{l=1}^{\beta_j}R(\gamma_{l,j})$ be a cycle. Assume that the image of $z$ under $\iota_j$ is a boundary, {\it i.e.}, that there is $y\in \cR^{\beta_{j+1}}$ such that $\tau(B_{j+1})(y)=z$. Write each component of $y$ as a formal sum of homogeneous components. Since every entry of $\tau(B_{j+1})$ is homogeneous and every component of $z$ is a polynomial, we can write $y=y_1+y_2$ where each component of $y_1$ is a polynomial and the order of each component of $\tau(B_{j+1})(y_2)$ is greater than the maximal degree of all components of $z$. It is clear now that $\tau(B_{j+1})(y_1)=z$ and $\tau(B_{j+1})(y_2)=0$. Since each component of $y_1$ is a polynomial, $z$ is a boundary in the complex $R^{\beta_{\bullet}}$. This implies the injectivity of $\iota_j$ on homology, proving part (a).

Now assume that $\dim_k\Tor^{\D}_j(R^{\tau}, M)<\infty$ for some $j \geq 0$. We show that $\iota_j$ is an isomorphism on homology. Suppose that $\widehat{z}$ is a cycle in $\cR^{\beta_j}$. Since the homology $h_j(R^{\beta_{\bullet}})$ is finite-dimensional, there is an integer $s_j$ such that, when restricted to graded pieces of degree greater than $s_j$, the complex $R^{\beta_{\bullet}}$ is exact at the $j$th spot. Write $\widehat{z}$ as a formal sum of homogeneous components. Each homogeneous component of sufficiently large degree, considered by itself, is a cycle in $R^{\beta_{\bullet}}$ (because the entries of the differential matrix $\tau(B_j)$ are homogeneous) and therefore is a boundary as well. The formal sum $\widehat{z'}$ of all such components is therefore a boundary in $\cR^{\beta_j}$, and all components of $\widehat{z} - \widehat{z'}$ are polynomials. That is, $\widehat{z} - \widehat{z'}$ belongs to $\oplus_{l=1}^{\beta_j}R(\gamma_{l,j})$. Thus $\widehat{z}$ differs by a boundary from a cycle in $\oplus_{l=1}^{\beta_j}R(\gamma_{l,j})$; {\it i.e.}, $\iota_j$ is surjective on homology. This completes the proof of part (b) and the proposition.
\end{proof}

\begin{prop}
\label{prop: de Rham Tor compatible with completion}
Let $M$ be a $\D$-module. For all integers $i \geq 0$, there are commutative diagrams
\[
\begin{tikzcd}
H^i_{\dR}(M) \arrow[d, "\kappa^{i}"] \arrow[r] & \Tor^{\D}_{n-i}(R^{\tau}, M) \arrow[d, "\iota_{n-i}"] \\
H^i_{\dR}(\cM) \arrow[r] & \Tor_{n-i}^{\cD}(\cR^{\tau}, \cM)
\end{tikzcd}
\]
where the vertical maps are induced by $M\to \cM$ and the horizontal ones are isomorphisms.
\end{prop}

\begin{proof}
The horizontal maps are the same as in \cite[Propositions 6.2.5.1, 6.2.5.2]{Bjork} which also assert that they are isomorphisms. It remains to show the commutativity, {\it i.e.}, that the maps in \cite[Propositions 6.2.5.1, 6.2.5.2]{Bjork} are compatible with completion. To this end, we will analyze these maps more closely. Viewing $\D$ as a $\D$-module, we can consider its de Rham complex $\Omega^{\bullet}(\D)$, whose differentials are right $\D$-linear and which can be regarded as a right $\D$-module resolution of $R^{\tau}$ (see the proof of \cite[Proposition 6.2.5.1]{Bjork}). Tensoring with the (left) $\D$-module $M$, we obtain a complex $\Omega^{\bullet}(\D) \otimes_{\D} M$ of $k$-spaces whose $i$th cohomology space (if the complex is indexed cohomologically) is both $H^i_{\dR}(M)$ (since $\Omega^{\bullet}(\D) \otimes_{\D} M$ can be canonically identified with the de Rham complex $\Omega^{\bullet}(M)$ of $M$) and $\Tor_{n-i}^{\D}(R^{\tau}, M)$ (since we can calculate this $\Tor$ group using the free resolution $\Omega^{\bullet}(\D) \rightarrow R^{\tau}$ of the first variable). Repeating this reasoning over the formal power series ring, we see that the $i$th cohomology space of the complex $\Omega^{\bullet}(\cD) \otimes_{\cD} \cM$ is simultaneously $H^i_{\dR}(\cM)$ and $\Tor_{n-i}^{\cD}(\cR^{\tau}, \cM)$. 

Recall that the natural map of complexes 
\[
\Omega^{\bullet}(M) = \Omega^{\bullet}(\D) \otimes_{\D} M \rightarrow \Omega^{\bullet}(\cD) \otimes_{\cD} \cM = \Omega^{\bullet}(\cM)
\]
of $k$-spaces, which we have denoted by $\kappa^{\bullet}$, induces the maps $\kappa^i$ on cohomology. Choose a free resolution $F_{\bullet}$ of $M$ as a $\D$-module. As in Proposition \ref{completion of Tor}, we obtain a chain map
\[
\iota_{\bullet}: R^{\tau} \otimes_{\D} F_{\bullet} \rightarrow \cR^{\tau} \otimes_{\cD} \widehat{F_{\bullet}}
\]
inducing the maps $\iota_j$ on homology.
Now consider the totalized tensor product complexes $(\Omega^{\bullet}(\D) \otimes_{\D} F_{\bullet})_{\bullet}$ and $(\Omega^{\bullet}(\cD) \otimes_{\cD} \widehat{F_{\bullet}})_{\bullet}$ (observe that there is a natural map of complexes of $k$-spaces from the former totalized complex to the latter). We obtain a diagram
\[
\begin{tikzcd}
\Omega^{\bullet}(\D) \otimes_{\D} M \arrow[r] \arrow[d, "\kappa^{\bullet}"] & (\Omega^{\bullet}(\D) \otimes_{\D} F_{\bullet})_{\bullet} \arrow[d] & R^{\tau} \otimes_{\D} F_{\bullet} \arrow[l] \arrow[d, "\iota_{\bullet}"] \\
\Omega^{\bullet}(\cD) \otimes_{\cD} \cM \arrow[r] & (\Omega^{\bullet}(\cD) \otimes_{\cD} \widehat{F_{\bullet}})_{\bullet} & \cR^{\tau} \otimes_{\cD} \widehat{F_{\bullet}} \arrow[l]
\end{tikzcd}
\]
where all four horizontal arrows are quasi-isomorphisms (by the balancing of $\Tor$, \cite[Theorem 2.7.2]{WeibelHA}) and both squares are commutative. From this, by passing to (co)homology, it follows that we have commutative diagrams
\[
\begin{tikzcd}
H^i_{\dR}(M) \arrow[d, "\kappa^{i}"] \arrow[r] & \Tor^{\D}_{n-i}(R^{\tau}, M) \arrow[d, "\iota_{n-i}"] \\
H^i_{\dR}(\cM) \arrow[r] & \Tor_{n-i}^{\cD}(\cR^{\tau}, \cM)
\end{tikzcd}
\]
of $k$-spaces for all $i \geq 0$.
\end{proof}

\begin{proof}[Proof of Theorem \ref{dR coh of completion of graded}]
It follows from Proposition \ref{completion of Tor} that for all $i \geq 0$, $\iota_{n-i}$ is an isomorphism in the diagram in Proposition \ref{prop: de Rham Tor compatible with completion}. Since the horizontal arrows in that diagram are also isomorphisms, so is $\kappa^i$, completing the proof.
\end{proof}

\begin{remark}
As shown in \cite{SwitalaZhangGradedDual}, there are indeed non-holonomic graded $\D$-modules whose de Rham cohomology spaces are all finite dimensional ({\it e.g.} graded Matlis dual of a graded holonomic $\D$-module).
\end{remark}

\begin{remark}\label{completion of Ext}
If one assumes that $M$ is graded and {\it holonomic}, then there is a short proof that the two $k$-spaces in Theorem \ref{dR coh of completion of graded} have the same dimension, using recent results from \cite{LyubeznikExtD} and \cite{SwitalaZhangGradedDual}. To thi send, let $M$ be a graded holonomic $\D$-module. Then by Proposition \ref{completion is holonomic}, $\cM$ is also holonomic. Therefore, both $H^i_{\dR}(M)$ and $H^i_{\dR}(\cM)$ are finite-dimensional $k$-spaces, so it suffices to check that $H^i_{\dR}(M)$ and $H^i_{\dR}(\cM)$ have the same dimension by \cite[Propositions 6.2.5.1, 6.2.5.2]{Bjork}. By \cite[Theorem 5.3]{SwitalaZhangGradedDual}, $H^i_{\dR}(M)$ and $\Ext^{n-i}_{\D}(M, E)$ have the same dimension; by \cite[Theorem 1.3]{LyubeznikExtD}, $H^i_{\dR}(\cM)$ and $\Ext^{n-i}_{\cD}(\cM, E)$ have the same dimension. (Here $E$ is the injective hull of $k=R/\mm$ as an $R$-module; $E$ is also an $\cR$-module and $\cR \otimes_R E = E$, so both $E$s are the same, and $E$ is also the injective hull of $k=\cR/\cmm$ as an $\cR$-module.) Therefore we need only show that $\Ext^j_{\D}(M, E) \cong \Ext^j_{\cD}(\cM, E)$ as $k$-spaces for all $j$. This is easy to see directly, by taking a graded free $\D$-module resolution $F_{\bullet} \rightarrow M$ and using $F_{\bullet}$ (resp. $\widehat{F_{\bullet}}$) to compute the $\Ext$ groups: the complexes $\Hom^{\bullet}_{\D}(F_{\bullet}, E)$ and $\Hom^{\bullet}_{\cD}(\widehat{F_{\bullet}}, E)$ are the same, so their cohomology spaces coincide.
\end{remark}

\bibliographystyle{plain}

\bibliography{masterbib}

\end{document}